\documentclass[11pt]{amsart}

\usepackage[utf8]{inputenc}
\usepackage{enumerate}
\usepackage[vmargin=2cm, hmargin=2cm]{geometry}
\usepackage{xcolor}
\usepackage{tikz}

\usepackage{parskip}
\usepackage{graphicx}

\newcommand{\dd}{\mathrm{d}}
\newcommand{\DD}{\mathrm{D}}
\newcommand{\sq}[1]{[\![{#1}]\!]}
\newcommand{\sg}[1]{\langle{#1}\rangle}
\newcommand{\HH}{\mathcal{H}}

\DeclareMathOperator{\Ap}{Ap}
\DeclareMathOperator{\F}{F}

\newtheorem{theorem}{Theorem}
\newtheorem{lemma}[theorem]{Lemma}
\newtheorem{corollary}[theorem]{Corollary}

\theoremstyle{definition}
\newtheorem{dfn}[theorem]{Definition}

\theoremstyle{remark}
\newtheorem{remark}[theorem]{\bf Remark}

\title{On the number of \textsf L-shapes in embedding dimension four}
\author{F. Aguil\'o-Gost}
\address{Departament de Matem\`atica Aplicada IV\\ Univ Polit\`ecnica de Catalunya\\ Barcelona, Espa\~na}
\thanks{Supported by MTM2011-28800-C02-01 and 2009SGR1387}
\email{matfag@ma4.upc.edu}

\author{P. A. Garc\'ia-S\'anchez}
\address{Departamento de \'Algebra\\ Universidad de Granada\\
Granada, Espa\~na} 
\email{pedro@ugr.es}
\thanks{Supported by MTM2010-15595, FQM-343, FQM-5849 and FEDER funds; part of this work was done during a visit of the second author to the Universidad de Almer\'{\i}a supported by the `plan propio' of this university}   

\author{D. Llena}
\address{Departamento de Matem\'aticas\\ Universidad de Almer\'{\i}a\\ Almer\'{\i}a, Espa\~na}
\email{dllena@ual.es}
\thanks{Supported by MTM2010-15595, FQM-343 and FEDER funds}

\begin{document}
\maketitle

\begin{abstract}
\textit{Minimum distance diagrams}, also known as \textit{\textsf{L}--shapes}, have been used to study some properties related to \textit{weighted Cayley digraphs} of \textit{degree} two and \textit{embedding dimension three numerical semigroups}. In this particular case, it has been shown that these discrete structures have at most two related  \textsf{L}--shapes. These diagrams are proved to be a good tool for studing \textit{factorizations} and the \textit{catenary degree} for semigroups and \textit{diameter} and \textit{distance} between vertices for digraphs.

This maximum number of \textsf{L}--shapes has not been proved to be kept when increasing the degree of digraphs or the embeding dimension of semigroups. In this work we give a family of embeding dimension four numerical semigroups $S_n$, for odd $n\geq5$, such that the number of related \textsf{L}--shapes is $\frac{n+3}2$. This family has her analog to weighted Cayley digraphs of degree three.

Therefore, the number of \textsf{L}--shapes related to numerical semigroups can be as large as wanted when the embedding dimension is at least four. The same is true for weighted Cayley digraphs of degree at least three. This fact has several implications on the combinatorics of factorizations for numerical semigroups and minimum paths between vertices for weighted digraphs.

\vspace*{20pt}


\end{abstract}

\section{Introduction}
Minimum Distance Diagrams (MDD for short) have been used in different discrete structures to study several optimization problems. Most known examples of this use are metrical optimization problems in Cayley digraphs on cyclic finite Abelian groups and several questions in numerical semigroups. Frobenius number computation, factorization related properties and the study of Ap\'{e}ry sets are some applications in the latter example.

A \emph{Cayley digraph} $G=C(N;s_1,\ldots,s_k;p_1,\ldots,p_k)$ on the cyclic finite Abelian group $\mathbb{Z}_N$ generated by the generator set $B=\{s_1,\ldots,s_k\}\subset\mathbb{Z}_N\setminus\{0\}$, $\gcd(N,s_1,\ldots,s_k)=1$,  is a directed graph with set of vertices $V=\mathbb{Z}_N$ and set of arcs $A=\left\{m\stackrel{p_i}{\longrightarrow}(m+s_i)\pmod{N} \mid m\in V,1\leq i\leq k\right\}$, where $p_i$ is the weight of the arc defined by $s_i$, $i\in\{1,\ldots,k\}$. The \emph{length} of a path in $G$ is the sum of the weights of the arcs in the path. A \textit{minimum path} from $m$ to $n$ is a connecting path from $m$ to $n$ with minimum length in $G$. The \textit{distance} from $m$ to $n$, $\dd(m,n)$, is the length of a minimum path from $m$ to $n$. The \textit{diameter} of $G$, $\DD(G)$, is the maximum of the distances between pairs of vertices in $G$. The metric on $G$ depends on the weights of his arcs.

Let us consider unit cubes in $\mathbb{R}^k$. Each unit cube $[i_1,i_1+1]\times[i_2,i_2+1]\times\dots\times[i_k,i_k+1]\in\mathbb{R}^k$ has integral coordinates $(i_1,\ldots,i_k)\in\mathbb{Z}^k$ and it is usually labelled with the vertex $i_1s_1+\dots+i_ks_k\pmod{N}$ (sometimes it is also labelled with his `weight' $i_1p_1+\dots+i_kp_k$). We denote the unit cube with coordinates $(i_1,\ldots,i_k)$ by $\sq{i_1,\ldots,i_k}$. Let $\leq$ be the usual partial ordering in $\mathbb{N}^k$. A unified definition of minimum distance diagrams was given by P.~Sabariego and F.~Santos in 2009 \cite[Definition~2.1]{SS:09} although other authors used this concept, see for instance Fiol et al. \cite{FYAV:87} and R\"odseth \cite{Ro:96}. Following the definition of \cite{SS:09}, a \emph{minimum distance diagram} $\HH$ related to $G$ is a connected set of $N$ unit cubes in $\mathbb{R}^k$ with different vertex label and the following two properties
\begin{itemize}
\item[(1)] if $u=\sq{i_1,\ldots,i_k}\in\HH$, then the weight $\|u\|=i_1p_1+\cdots+i_kp_k$ is minimum over all cubes with coordinates $(j_1,\ldots,j_k)\in\mathbb{N}^k$ fulfilling $j_1s_1+\cdots+j_ks_k\equiv i_1s_1+\cdots+i_ks_k\pmod{N}$,
\item[(2)] if $v=\sq{j_1,\ldots,j_k}$ is a cube with $(j_1,\ldots,j_k)\leq(i_1,\ldots,i_k)$, then $v\in\HH$.
\end{itemize}
It have been proved these diagrams are L-shaped regions of the plane (or rectangles) when $k=2$ (\cite{FYAV:87}). For this reason they are called \textit{\textsf L-shapes} when $k=2$ and \textit{hyper \textsf L-shapes} when $k\geq3$.

\begin{figure}[h]
\centering
\includegraphics[width=0.3\linewidth]{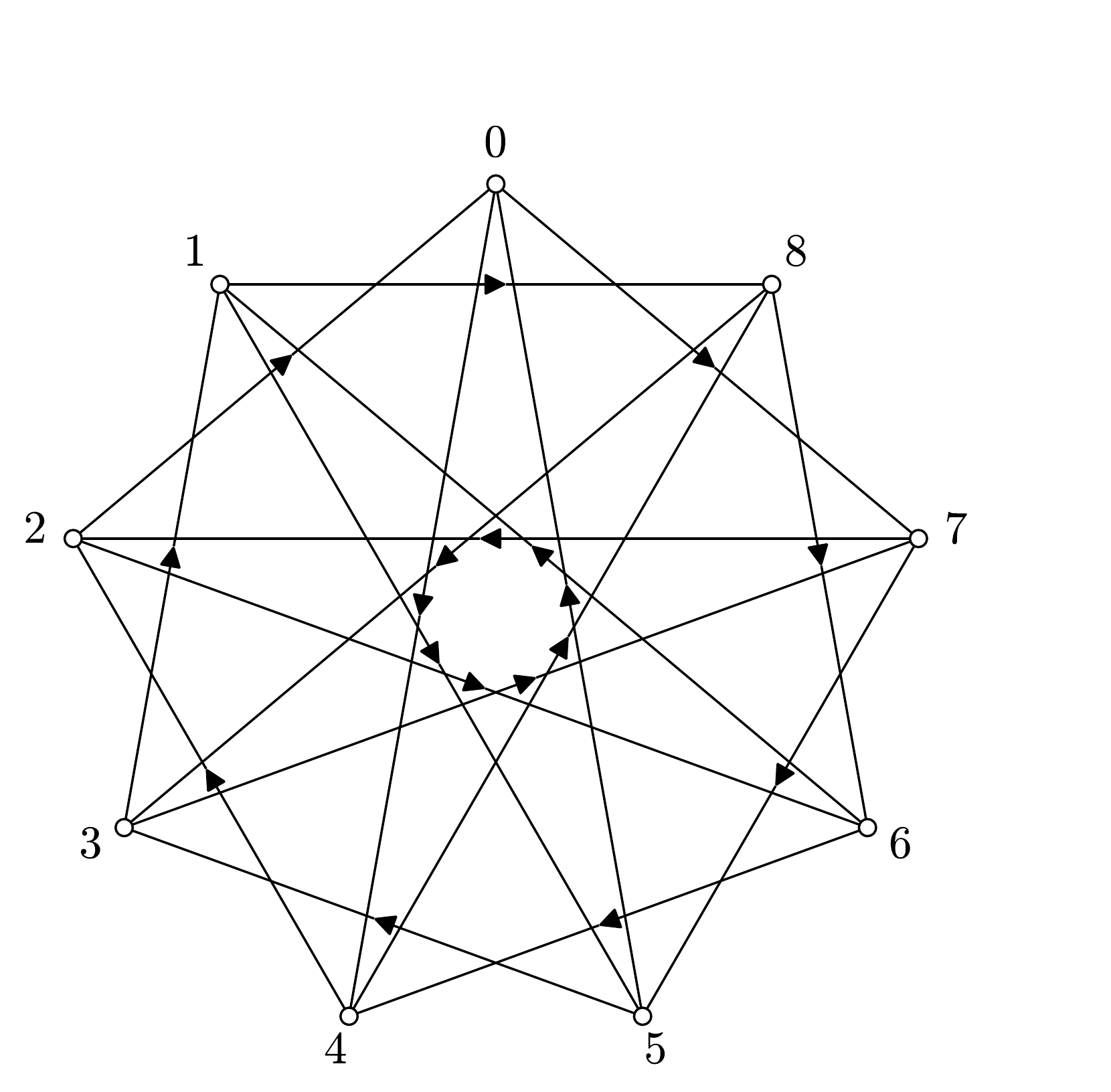}
\hspace*{0.05\linewidth}
\includegraphics[width=0.15\linewidth]{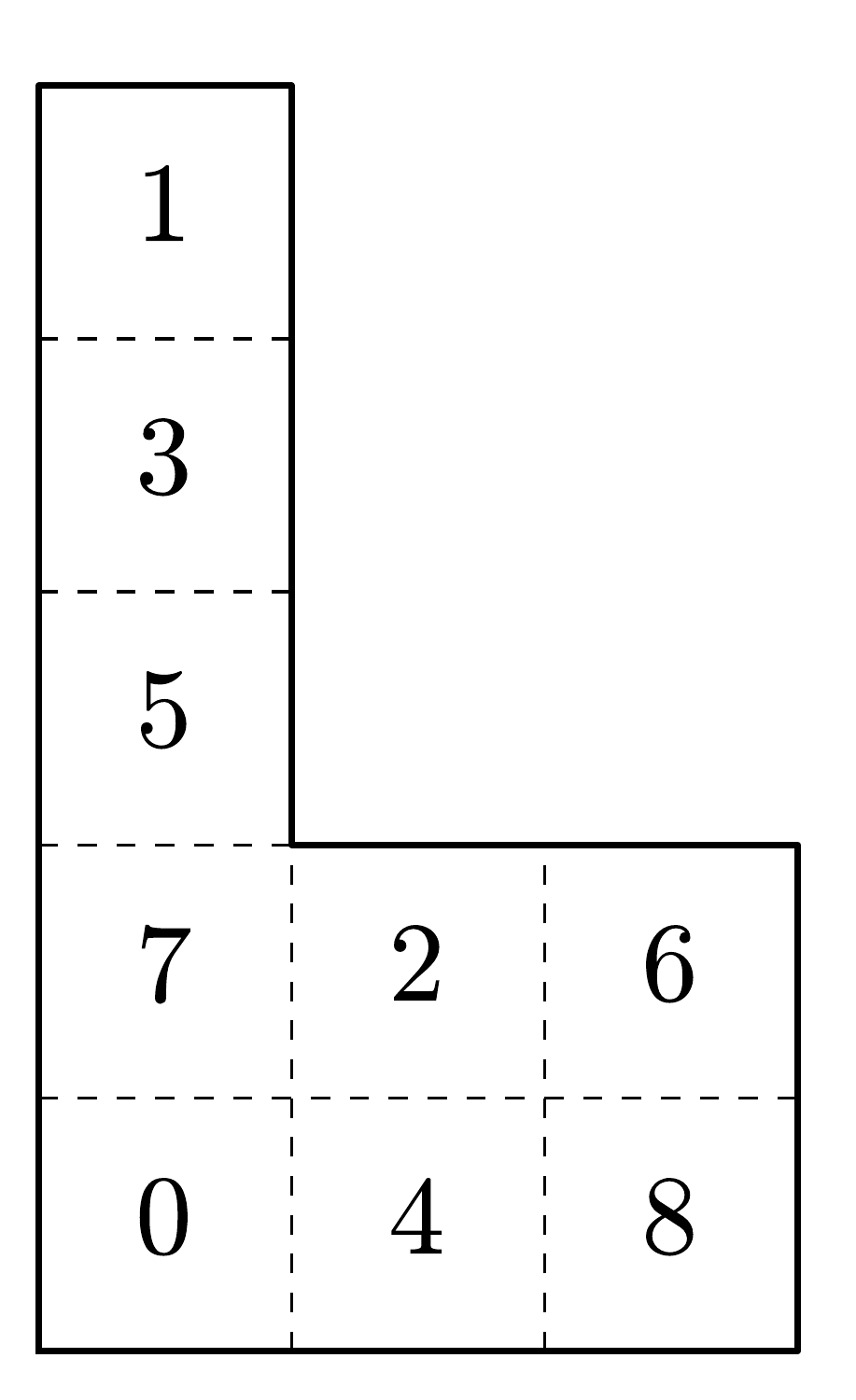}
\hspace*{0.1\linewidth}
\includegraphics[width=0.19\linewidth]{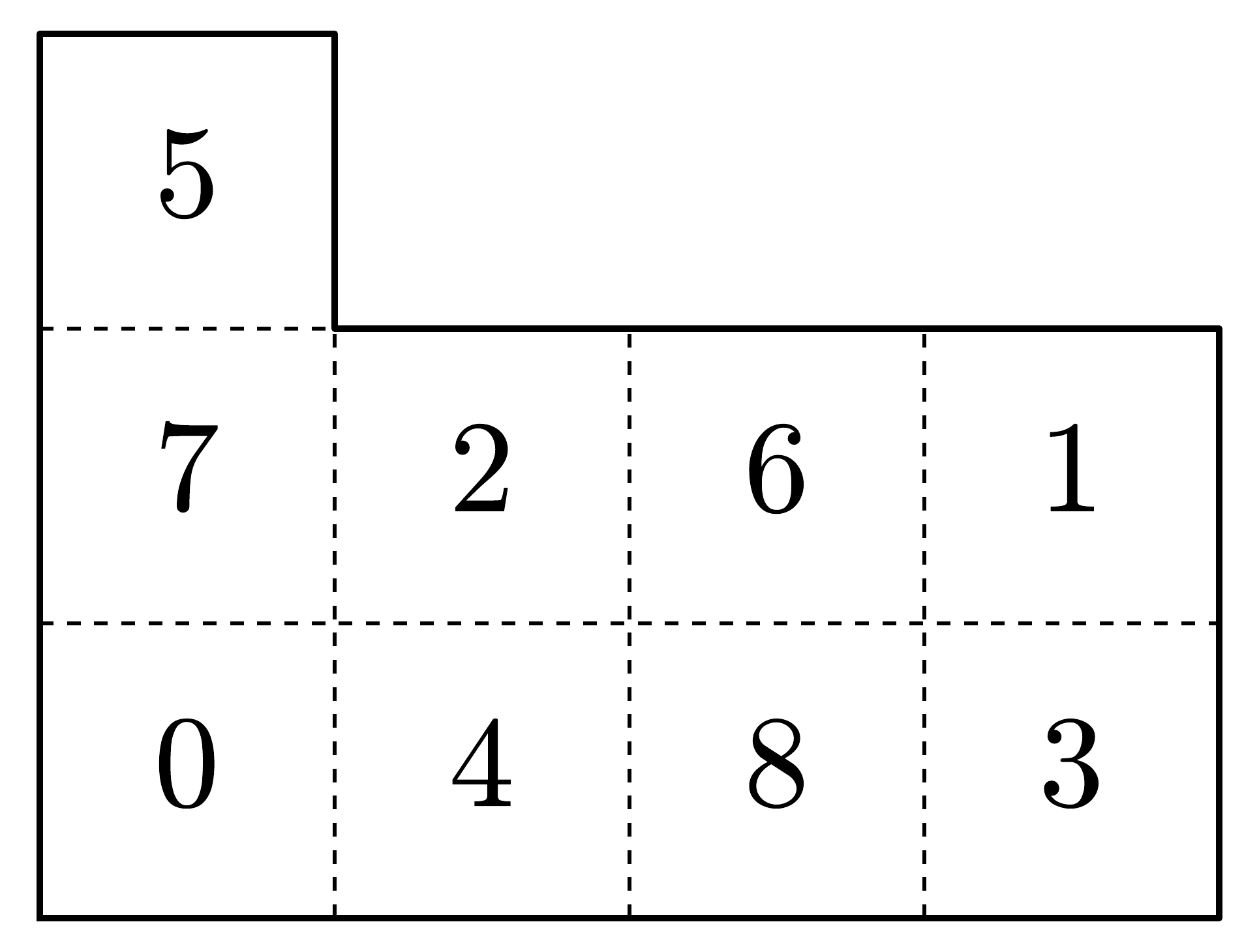}
\caption{$C(\mathbb{Z}_9;4,7;1,1)$ and two related MDD}
\label{fig:c947}
\end{figure}

Usually, problems in Graph Theory are stated in the non--weighted version of arcs, that is, $p_1=\cdots=p_k=1$. For instance, Figure~\ref{fig:c947} shows two minimum distance diagrams associated with $C(\mathbb{Z}_9;4,7;1,1)$. There is a significative difference between $k=2$ and $k\geq3$. When $k=2$, it has been shown that these digraphs have two related MDD at most. For $k=3$, Sabariego and Santos \cite{SS:09} gave an infinite family of digraphs with many associated MDDs. More precisely, given $t\not\equiv 0\pmod{3}$, set $m=2+t+t^2$; then, the digraph $G_t=C(m(m-1);1+m,1+mt,1+mt^2;1,1,1)$ has $3(t+2)$ associated MDDs. Taking $t=2$, the digraph $C(\mathbb{Z}_{56};9,17,33;1,1,1)$ has $12$ different associated MDDs that have been depicted in Figure~\ref{fig:c3D}.

\begin{figure}[h]
\centering
\includegraphics[width=0.125\linewidth]{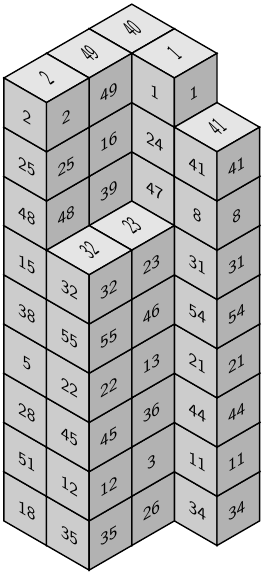}
\hspace*{0.07\linewidth}
\includegraphics[width=0.2\linewidth]{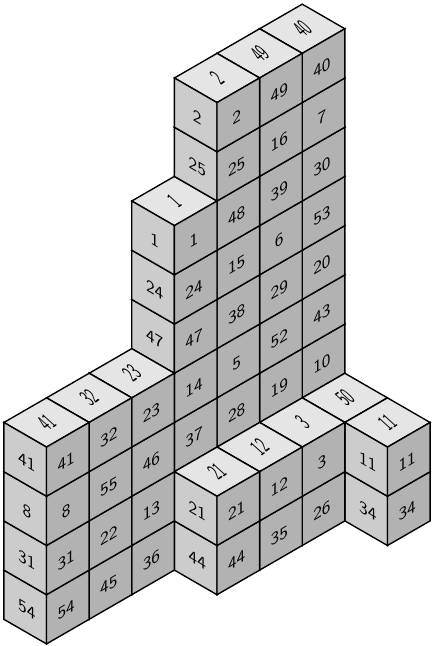}
\hspace*{0.07\linewidth}
\includegraphics[width=0.25\linewidth]{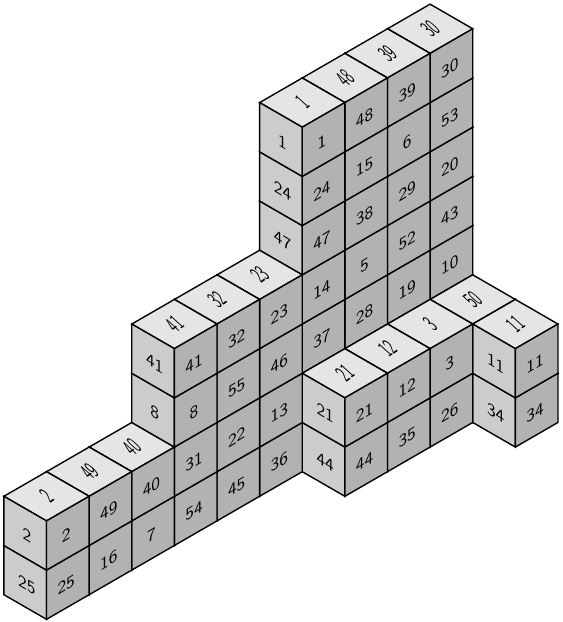}\\
\includegraphics[width=0.25\linewidth]{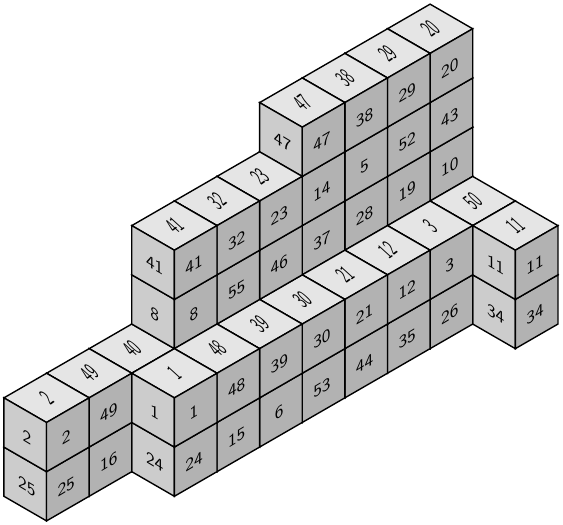}
\hspace*{0.07\linewidth}
\includegraphics[width=0.25\linewidth]{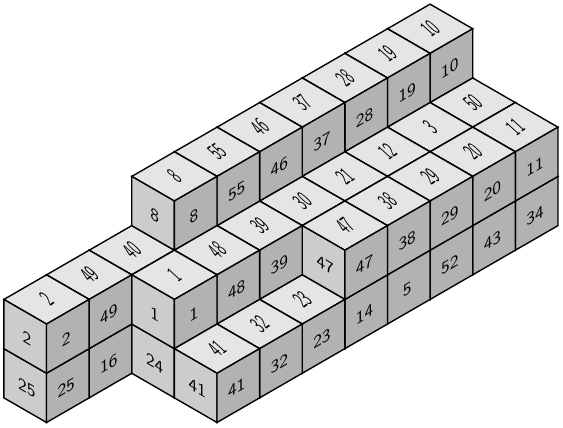}
\hspace*{0.07\linewidth}
\includegraphics[width=0.15\linewidth]{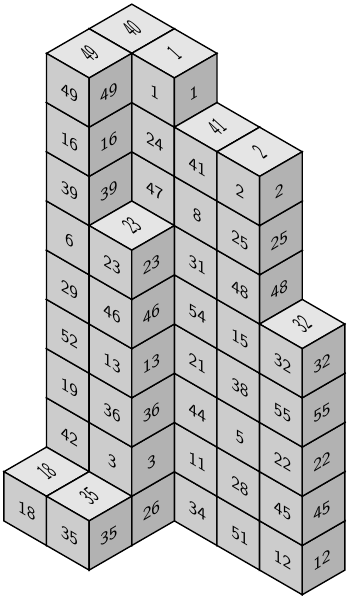}\\
\includegraphics[width=0.2\linewidth]{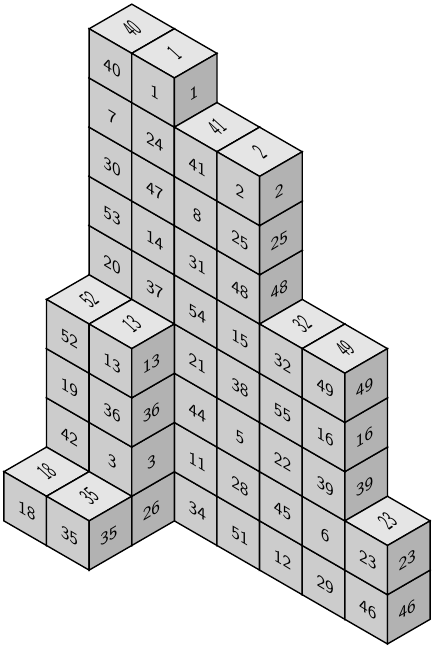}
\hspace*{0.05\linewidth}
\includegraphics[width=0.27\linewidth]{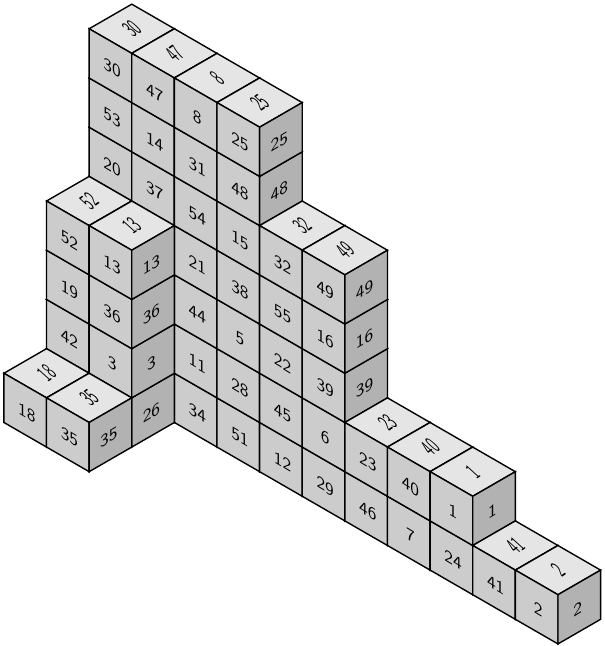}
\hspace*{0.05\linewidth}
\includegraphics[width=0.27\linewidth]{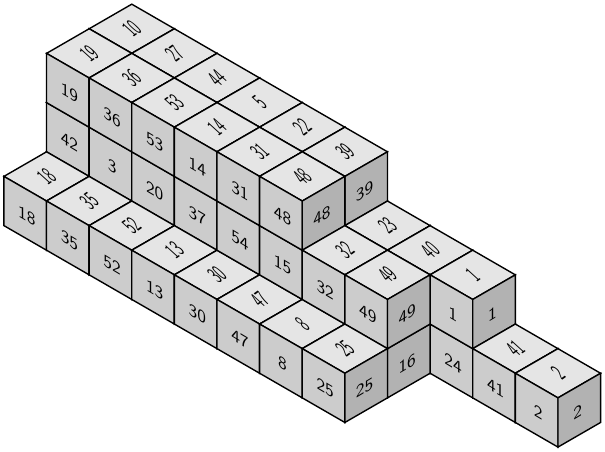}\\
\includegraphics[width=0.32\linewidth]{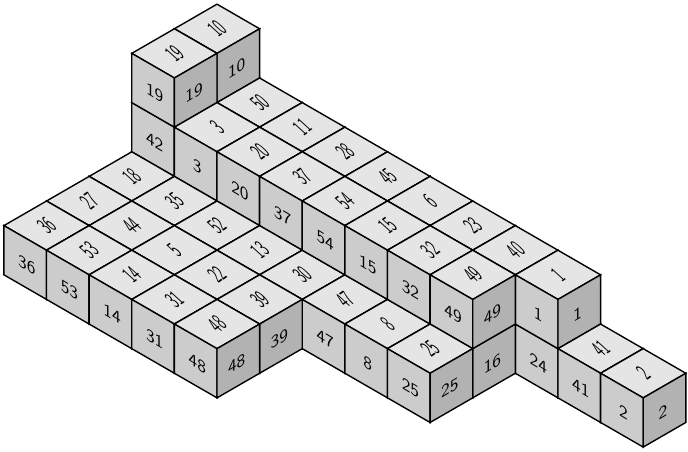}
\includegraphics[width=0.32\linewidth]{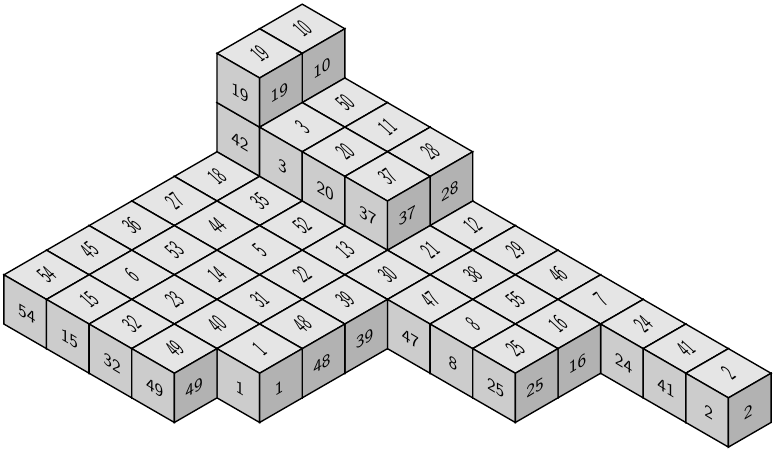}
\includegraphics[width=0.32\linewidth]{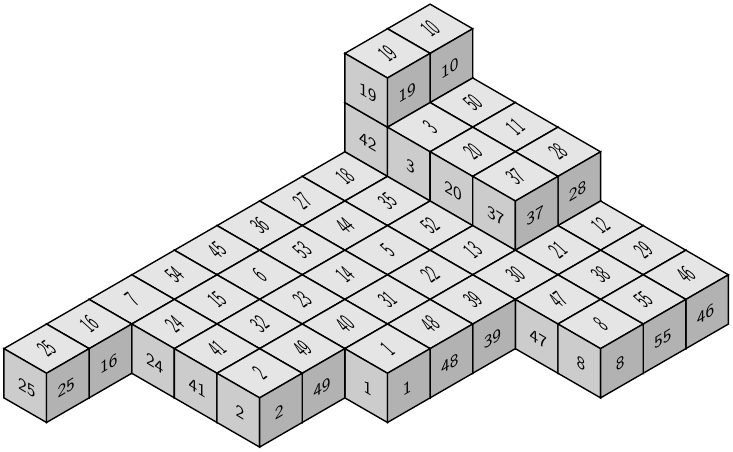}
\caption{$C(\mathbb{Z}_{56};9,17,33;1,1,1)$ and related MDDs}
\label{fig:c3D}
\end{figure}

Changing weights $\{p_1,\ldots,p_k\}$ provides a different metric on the digraph. Thus, the number of minimum distance diagrams eventually decreases. Non-weighted version of the digraph has more related MDDs than the weighted one. For instance, changing $p_1=p_2=1$ to $p_1=2$ and $p_2=3$, only the second diagram in Figure~\ref{fig:c947} is an MDD related to $C(\mathbb{Z}_9;4,7;2,3)$. Also taking $p_1=9$, $p_2=17$ and $p_3=33$ only the fifth and the last diagrams in Figure~\ref{fig:c3D} are MDDs related to $G_2=C(\mathbb{Z}_{56};9,17,33;9,17,33)$. Table~\ref{tab:nMDD} shows the number of MDDs resulting from taking the weights $p_1=1+m$, $p_2=1+mt$ and $p_3=1+mt^2$ in $G_t$. Important types of weights are $p_i=s_i$ for all $i$. This choice of weights can modelize some properties of numerical semigroups on  digraphs.

\begin{table}[h]
\centering
\begin{tabular}{|r||rrrrrrrrrrrrr|}\hline
$t$&2&4&5&7&8&10&11&13&14&16&17&19&20\\\hline
$N_t$&56&462&992&3306&5402&12432&17822&33672&44732&74802&94556&145542&177662\\
\# non--w.&12&18&21&27&30&36&39&45&48&54&57&63&66\\
\# w.&      2& 2& 2&2& 2& 2&12&2 & 2& 14& 2& 2& 2\\\hline
\end{tabular}
\vspace*{0.01\linewidth}
\caption{Number of related MDDs for weighted and non--weighted $G_t$, $t\not{\!\!\equiv}0\pmod{3}$}
\label{tab:nMDD}
\end{table}

Now, we look at numerical semigroups to translate the Minimum Distance Diagrams to this setting.

Given $n_1,\ldots,n_k\in\mathbb{N}$ with $\gcd(n_1,\ldots,n_k)=1$, the \emph{numerical semigroup} generated by $A=\{n_1,\ldots,n_k\}$ is the set $S=\sg{n_1,\ldots,n_k}=\{x_1n_1+\cdots+x_kn_k:~(x_1,\ldots,x_k)\in\mathbb{N}^k\}$. We use the notation $Ax=x_1n_1+\cdots+x_kn_k$ for $x\in\mathbb{N}^k$.

Let $S$ be a numerical semigroup minimally generated by $A$ (no proper subset of $A$ generates the same semigroup), and let $k$ be the cardinality of $A$, which is known as the \emph{embedding dimension} of $S$. 

The \emph{set of factorizations} of $s\in S$ is $\mathsf Z(s)=\{x\in \mathbb N^k \mid Ax=s\}$. For a subset $X$ of $S$, let $\mathsf Z(X)=\bigcup_{x\in X}\mathsf Z(x)$ (this union is disjoint).


For $a,b\in \mathbb Z$, we write $a\le_S b$ if $b-a\in S$. Let $C\subseteq S$. We say that $C$ is \emph{closed} if whenever $a\in C$ and $b\le_S a$, then $b\in C$. 

\begin{dfn}
Fixed $C$ a nonempty closed subset of $S$, we say that $L \subset \mathbb N^k$ is an \emph{\textsf L-shape} associated to $C$ if the following two properties hold
\begin{enumerate}[(C1)]
\item the map $x\mapsto Ax$ is a bijection from $L$ to $C$ ($\# (\mathsf Z(c)\cap L)=1$ for all $c\in C$),
\item if $x\in L$, then $y\in L$ for every $y\in\mathbb N^k$ with $y\le x$.
\end{enumerate}
\end{dfn}

A particular case of closed sets in a numerical semigroup are the Ap\'ery sets. Let $m\in S\setminus\{0\}$. The \emph{Ap\'ery set} of $m$ in $S$ is the set 
\[
\Ap(S,m)=\{ s\in S\mid s-m\not\in S\}.
\]
It can be easily shown that $\Ap(S,m)=\{w_0,\ldots, w_{m-1}\}$, where $w_i=\min \{s\in S\mid s\equiv i\pmod{m}\}$ for $i\in\{0,\ldots, m-1\}$. In particular, the cardinality of $\Ap(S,m)$ is $m$ (see for instance \cite[Lemma 2.4]{ns}).

Given a numerical semigroup $S=\sg{n_1,\ldots,n_k}$, let us consider the related digraph
\[
G_S=C(n_k;n_1,\ldots,n_{k-1};n_1,\ldots,n_{k-1}).
\]
Many metric properties of the digraph $G_S$ give information on the semigroup $S$. For instance, for $k=3$, generic properties of the sets of factorizations are studied in \cite{AG:2010}.

Let us denote the weight of the unit cube $u=\sq{i_1,\ldots,i_{k-1}}$ as $\|u\|=n_1i_1+\cdots+n_{k-1}i_{k-1}$. If $\HH$ is a minimum distance diagram associated with $G_S$, then it can be seen that
\[
\mathcal{W}_{\HH}=\{\|u\| \mid u\in\HH\}=\Ap(S,n_k).
\]
Thus, an \textsf L-shape related to $S$ is equivalent to a minimum distance diagram associated with $G_S$. It is also known that for embedding dimension three, $S$ admits at most two $L$-shapes. So, there is a natural question arising from this equivalence for embedding dimension larger than three: are there numerical semigroups with a number of related \textsf L-shapes as large as we want? As far as we know, there is no related work in the bibliography.

\section{A distinguished infinite family of $4$--semigroups}

We have implemented the construction of \textsf L-shapes in the \texttt{numericalsgps} (\cite{numericalsgps}) \texttt{GAP} (\cite{GAP4}) package. Computer evidence convinced us to look for a parameterized family of embedding dimension four numerical semigroups with as many \textsf L-shapes as desired. 

Bresinsky in \cite{bresinsky} gave a family of numerical semigroups with embedding dimension four with arbitrary large minimal presentations (embedding dimension three numerical semigroups admit minimal presentations with at most three elements; thus the analogy with our setting). Unfortunately all the elements in his family have exactly two \textsf L-shapes.

In order to look for primitive elements, we developed an algorithm in \cite{AGL:2014} that gave us some light to find the family of semigroups that we present in this paper.

Let $n$ be an odd integer greater than or equal to five. Set
\[T=\langle n,3n-2, 3n-1\rangle \hbox{ and } S=T\cup\{\mathrm F(T)\}=\langle n,3n-2,3n-1,\mathrm F(T)\rangle,\]
where $\F(T)$ denotes $\max(\mathbb Z\setminus T)$, the Frobenius number of $T$.

Observe that a minimal generating set for $T$ is $\{n,3n-2,3n-1\}$. Hence there exists an epimorphism $\phi:\mathbb N^{3}\to T$, $\phi(a,b,c)=an+b(3n-2)+c(3n-1)$. The kernel of $\phi$, $\ker\phi=\{ (\alpha,\beta)\in \mathbb N^{3}\times \mathbb N^{3}\mid \phi(\alpha)=\phi(\beta)\}$ is a congruence. A minimal generating system of $\ker\phi$ as a congruence is known as a \emph{minimal presentation} for $T$.
Minimal presentations turn out to be a key tool in the study of factorizations. 

In order to find a minimal presentation of $T$ we must find the least multiple of each generator that belongs to the semigroup spanned by the other two (see for instance \cite[Example 8.23]{ns}).

Clearly, $2(3n-1)=3n+(3n-2)$. Observe that if we look for the least multiple of $3n-2$ that belongs to $\langle n,3n-1\rangle$, we have to solve the equation $a(3n-2)=bn+c(3n-1)$. Thus we are looking for $a,b,c\in\mathbb N$, $a\neq 0$ such that $(3a-b-3c)n=2a-c$. Hence we must solve 
\[
\begin{cases}
3a-b-3c=k,\\
2a-c=kn,
\end{cases}
\]
with $k\in \mathbb N$. We get the parametrized solutions $a=\frac{kn+c}2$ and $b=\frac{k(3n-2)-3c}{2}$. For $k=0$ there is no nonnegative solution to the equations, and thus the least possible $a$ is reached for $k=1$, and since $n$ is odd, in order to get $a\in\mathbb N$, $c$ cannot be zero. Hence the least possible value of $a$ is $\frac{n+1}2$, and then $b=\frac{3n-5}2\in \mathbb N$ and $c=1$. So we already have two relations among the generators:
\[
\begin{array}{rcl}
2(3n-1) & = & 3n+1(3n-2),\\
\frac{n+1}2 (3n-2)& = &\frac{3n-5}2 n + 1(3n-1).
\end{array}
\]
In light of \cite[Lemma 10.19]{ns}, the third relation can be obtained from these two by ``adding'' them together:
\[
\frac{3n+1}2 n = \frac{n-1}2 (3n-2) + 1(3n-1).
\]
Therefore, a minimal presentation for $T$ is 
\begin{equation}\label{minimalpresentation}
\left\{ ((0,0,2),(3,1,0)), \left(\left(0,\frac{n+1}2,0\right),\left(\frac{3n-5}2,0,1\right)\right),\left(\left(\frac{3n+1}2,0,0\right),\left(0,\frac{n-1}2,1\right)\right)\right\}.
\end{equation}
From \cite[Proposition 2.20 and Lemma 10.20]{ns}, we obtain that the set of pseudo-Frobenius numbers of $S$, $\mathrm{PF}(T)=\{ z\in \mathbb Z\mid z+T\setminus\{0\} \subseteq T\}$, is
\[
\mathrm{PF}(T)=\left\{\frac{3n-7}2n+1, \frac{3n-7}2n+2\right\}.
\]
In particular $\F(T)=\frac{3n-7}2n+2$ and $\F(T)-1\not\in T$. This implies that $\F(S)=\F(T)-1$ and  as a consequence of this, $\max\Ap(S,\F(T))=2\F(T)-1=n(3n-7)+3$ (\cite[Proposition 2.12]{ns}).

\subsection{Factorizations of the elements of the Ap\'ery set}
In this section we describe what are the factorizations of the elements of $\Ap(S,\F(T))$. We are going to use extensively the fact that $\Ap(S,\F(T))$ is a closed set, as it has been remarked before. Also every element $s$ in $\Ap(S,\F(T))$ is in $T$, and thus we identify the set $\mathsf Z(s)$ with a subset of $\mathbb N^3$; indeed $\mathsf Z(s)$ is in one-to-one correspondence with $\phi^{-1}(s)$.

\begin{lemma}\label{fact-en-ap}
Let $s\in \Ap(S,\F(T))$. There exists exactly one factorization $(x,y,z)\in\mathsf Z(s)$ such that $z<2,$ $y<\frac{n+1}{2}$ and $x<\frac{3n-1}{2}$.
\end{lemma}
\begin{proof}
We can use the minimal presentation of $T$ to obtain one factorization of $s$ with the third coordinate less than two and the second less than $\frac{n+1}{2}$. Notice that as $n\frac{3n+1}2- \F(T)= 4n-2\in S$, we have $n\frac{3n+1}2\not\in\Ap(S,\F(T))$, and so the third relation is never used on the factorizations of $s$. As $n\frac{3n-1}2- \F(T)= 3n-2\in S$, the first coordinate must be less than $\frac{3n-1}2$.

Now assume that there is another $(x',y',z')\in \mathsf Z(s)$ with $z'<2$, $x'<\frac{3n-1}2$ and $y<\frac{n+1}2$. From the definition of minimal presentation there should be a chain of reductions going from $(x,y,z)$ to $(x',y',z')$ by using the relations in the minimal presentation. We already know that the third relation cannot be used. Also as $z,z'<2$ and $y,y'<\frac{n+1}2$, the only possibility is that either $(3,1,0)< (x,y,z)$ or $\left( \frac{3n-5}2,0,1\right)< (x,y,z)$. If $\left( \frac{3n-5}2,0,1\right)< (x,y,z)$, then $z=1$ and $\left( x-\frac{3n-5}2,y+\frac{n+1}2,0\right)\in\mathsf Z(s)$. Also $x<\frac{3n-1}2$, whence $x-\frac{3n-5}2<\frac{3n-1}2-\frac{3n-5}2=2$. So to meet this new factorization we cannot apply the first relation, which means that we can only, eventually, use the second one obtaining always factorizations with second coordinate greater than $\frac{n+1}2$. Assume now that $(3,1,0)<(x,y,z)$. Then $(x-3,y-1,z+2)\in \mathsf Z(s)$. Again, as $x<\frac{3n-1}2$, $x-3<\frac{3n-5}2$, and $y-1<\frac{
n+1}2$. So we cannot apply here the second relation. This means that we could only apply here the first one, obtaining in any case factorizations with last coordinate greater than two.
\end{proof}

We can define an injective mapping from $\Ap(S,\F(T))$ to $\mathbb N^3$ that assigns to every $s\in \Ap(S,\F(T))$ the only factorization $(x,y,z)$ fulfilling the conditions of Lemma \ref{fact-en-ap}. Let us denote this map by \[\mathrm{nf}:\Ap(S,\F(T))\to \mathbb N^3.\]
For $s\in\Ap(S,\F(T))$, we will say that $\mathrm{nf}(s)$ is the \emph{normal form} of $s$. As usual, given $X\subseteq \Ap(S,\F(T))$, we write $\mathrm{nf}(X)=\{\mathrm{nf}(s)\mid s\in X\}$.

\begin{lemma}\label{endings}
Under the standing hypothesis, 
\begin{multline*}\left\{\left(\frac{3n-3}{2}\right)n+\left(\frac{n-3}{2}\right)(3n-2),\quad n+\left(\frac{n-1}{2}\right)(3n-2),\right.\\ \left(\frac{3n-5}{2}\right)n+\left(\frac{n-5}{2}\right)(3n-2)+(3n-1),\\
\left.\left(\frac{3n-3}{2}\right)n+\left(\frac{n-7}{2}\right)(3n-2)+(3n-1)  
 \right\}\subseteq \Ap(S,\F(T)).
\end{multline*}
\end{lemma}
\begin{proof}
Clearly, $\left(\frac{3n-3}{2}\right)n+\left(\frac{n-3}{2}\right)(3n-2)\in S$, and $\left(\frac{3n-3}{2}\right)n+\left(\frac{n-3}{2}\right)(3n-2)-\mathrm F(T)=\mathrm F(T)-1\not\in S$.

As $n+\left(\frac{n-1}{2}\right)(3n-2)-\F(T)=2n-1\not\in S$, it easily follows that $n+\left(\frac{n-1}{2}\right)(3n-2)$ is in the Apéry set.

From identity $\left(\frac{3n-5}{2}\right)n+\left(\frac{n-5}{2}\right)(3n-2)+(3n-1)-\mathrm F(T)=\mathrm F(T)-n$, it follows the third element also belongs to this Ap\'ery set.

Finally, the last element in the list is in the Apéry since $\left(\frac{3n-3}{2}\right)n+\left(\frac{n-7}{2}\right)(3n-2)+(3n-1)-\F(T)=\F(T)-(3n-2)\not\in S$.

\end{proof}


Let 
\begin{multline*}
F=\left\{(x,y,0)\in \mathbb N^3 \mid x\le \frac{3n-3}2, y\leq\frac{n-3}2\right\}\cup \left\{(x,y,0)\in \mathbb N^3\mid x\le 1, y=\frac{n-1}2\right\} \\ 
 \cup \left\{(x,y,1)\in \mathbb N^3 \mid x\leq \frac{3n-3}2, y\leq\frac{n-7}2\right\}\cup \left\{(x,y,1)\in\mathbb N^3\mid x\leq \frac{3n-5}2, y=\frac{n-5}2\right\}.
\end{multline*}

We will denote respectively $F_1$, $F_2$, $F_3$ and $F_4$ the four sets in the above representation of $F$. 

\begin{lemma}\label{wholeApery}
The map $\mathrm{nf}:\Ap(S,\F(T))\to  F$ is a bijection.
\end{lemma}
\begin{proof}
By using Lemma \ref{endings} and the fact that $\Ap(S,\F(T))$ is closed, we have $xn+y(3n-2)\in \Ap(S,\F(T))$ for all $(x,y,0)\in F_1\cup F_2$. Also, the elements $xn+y(3n-2)+(3n-1)$ belong to the Ap\'ery $\Ap(S,\F(T))$, with $(x,y,1)\in F_3\cup F_4$. 

All the elements we have obtained so far are different by Lemma \ref{fact-en-ap}. Counting them all, we get 
\[
\left(\frac{3n-1}2\right)\left(\frac{n-1}2\right)+2+\left(\frac{3n-1}2\right)\left(\frac{n-5}2\right)+\frac{3n-3}2=\frac{3n^2-7n}2+2=\F(T).
\]
And as the cardinality of $\Ap(S, \F(T))$ is precisely $\F(T)$, we obtain the desired result.
\end{proof}

Observe that $F$ is an \textsf L-shape associated with $\Ap(S,\F(T))$. 


%
\subsection{Computing the number of factorizations of elements in $\Ap(S,\F(T))$}

\begin{lemma}\label{lem:unafact}
An element $s\in \Ap(S,\F(T))$ has only one factorization if and only if $\mathrm{nf}(s)$ satisfies one of the following disjoint conditions:
\begin{enumerate}[1)]
\item $\mathrm{nf}(s)=(x,0,0)$ with $0\leq x\leq \frac{3n-3}2$,
\item $\mathrm{nf}(s)=(x,0,1)$ with $0\leq x\leq \frac{3n-7}{2}$,
\item $\mathrm{nf}(s)=(2,y,0)$ with $1\leq y\leq \frac{n-3}{2}$,
\item $\mathrm{nf}(s)=(x,y,0)$ with $0\leq x\leq 1$ and $1\leq y\leq \frac{n-1}{2}$,
\item $\mathrm{nf}(s)=(x,y,1)$ with $0\leq x\leq 2$ and $1\leq y\leq \frac{n-5}{2}$.
\end{enumerate}
\end{lemma}
\begin{proof}
Notice that we are looking for elements in $F$ such that they are not bigger than or equal to (with respect to the usual partial ordering in $\mathbb N^3$) any of the six factorizations involved in the minimal presentation of $S$, \eqref{minimalpresentation}. So the sufficiency is clear.

As $((0,0,2),(3,1,0))$ is in the minimal presentation, we have that $z<2$. We begin with the case $y=0$.
 \begin{itemize}
 \item If $z=0$, as we have elements with only one factorization, the factorization $(x,0,0)$  must not be bigger than or equal to $(\frac{3n+1}{2},0,0)$. If this were not the case, we could apply the third element of the minimal presentation. Then we are choosing elements in $F_1$, so only can take $x\leq\frac{3n-3}{2}$. 
 \item If $z=1$, the factorization $(x,0,1)$ must be located below $(\frac{3n-5}{2},0,1)$ since, otherwise, $\left(x-\frac{3n-5}2,\allowbreak \frac{n+1}2,\allowbreak 0\right)$ would be another factorization of $\mathrm{nf}(s)$ (in this case we are lying in $F_3$, but in this case this is not relevant).
\end{itemize}
Now we look at the case $y>0$. In this setting we have $x<3$. We distinguish either $z=0$ or $z=1$. For $z=0$ we get two subcases.
\begin{itemize}
 \item If $x=2$, these elements are in $F_1$, and consequently we have $y\leq\frac{n-3}2$.
 \item If $x\leq 1$, we are choosing elements in $F_2$, so we have $y\leq\frac{n-1}{2}$. 
 \end{itemize}
Finally, if $y\geq 1$ and $z=1$, we are taking elements in $F_3\cup F_4$, whence $y\leq\frac{n-5}2$.
\end{proof}

 \begin{corollary}
 There exist $6n-13$ elements in $\Ap(S,\F(T))$ with only one factorization.
 \end{corollary}
 \begin{proof}
 We only need to sum the elements in each of the five items from Lemma~\ref{lem:unafact}:
 \[
 \frac{3n-1}2+\frac{3n-5}2+\frac{n-3}2+2\frac{n-1}2+3\frac{n-5}2=6n-13.\qedhere
 \]
 \end{proof}
 
Define $M_i$ as the set of elements in $\Ap(S,\F(T))$ having $i$ factorizations, i.e.
\[
M_i=\{ s\in \Ap(S,\F(T))\mid~\#\mathsf Z(s)=i\}.
\]

\begin{lemma}
Let $i$ be a positive integer such that $2\leq i\leq\frac{n-3}2$. An element $s\in M_i$ if and only if $\mathrm{nf}(s)=(x,y,z)$ satisfies one of the following disjoint conditions:
\begin{enumerate}[1)]\label{casesgeneral}
\item $y=i-2,$ $z=1$ and $\frac{3n-5}2\leq x \leq \frac{3n-3}2$,
\item $y=i-1$, $z=0$ and $3(i-1)\leq x\leq \frac{3n-3}2$,
\item $y=i-1$, $z=1$ and $3(i-1)\leq x\leq \frac{3n-7}2$,
\item $i\leq y\leq\frac{n-3}2$, $z=0$ and $3(i-1)\leq x \leq 3i-1$, 
\item $i\leq y\leq\frac{n-5}2$, $z=1$ and $3(i-1)\leq x \leq 3i-1$. 
\end{enumerate}
\end{lemma}
\begin{proof}
To construct the set of all factorizations of an element $s\in\Ap(S,\F(T))$, we can start with $\mathrm{nf}(s)=(x,y,z)$, and then, by using the elements of the minimal presentation, as times as possible, we can find the remaining factorizations for $s$ (this is a consequence of \cite[Lemma~8.4]{ns}). As we will take elements in $\Ap(S,\F(T))$, then, as we have seen in the proof of Lemma~\ref{fact-en-ap}, only $((0,0,2),(3,1,0))$ and $\left(\left(0,\frac{n+1}2,0\right),\left(\frac{3n-5}2,0,1\right)\right)$ can be used in this construction. We denote them as $\mathrm{mp}_1$ and $\mathrm{mp}_2$, respectively. We will refer to operation $\mathrm{mp}_i$ when we obtain a new factorization of an element by using the relation $\mathrm{mp}_i$. For instance, from the factorization $(0,0,2)$, by using $\mathrm{mp}_1$, we obtain $(3,1,0)$.

Starting with $(x,y,z)=\mathrm{nf}(s)$, the operation $\mathrm{mp}_1$ only can be applied by subtracting $(3,1,0)$ and adding $(0,0,2)$, since the elements in $F$ have $z\leq 1$. Analogously, operation $\mathrm{mp}_2$ only can be applied by subtracting $\left(\frac{3n-5}2,0,1\right)$ and adding $\left(0,\frac{n+1}2,0\right)$; we can not subtract $\left(0,\frac{n+1}2,0\right)$ because the second coordinate of the elements in $F$ are less than or equal to $\frac{n-1}{2}$. Also, the operation $\mathrm{mp}_2$ can only be applied to elements on $F$ with the first coordinate $x=\frac{3n-3}2$ or $x=\frac{3n-5}2$ and the third coordinate $z=1$.

We are going to see that in the tree of factorizations obtained by applying $\textrm{mp}_1$ and $\textrm{mp}_2$, whenever we apply $\mathrm{mp}_2$, we encounter a leaf, that is, we cannot obtain new factorizations following this path. Actually the tree (rooted in $\mathrm{nf}(s)$) would have the following two possible shapes.

\begin{center}
\begin{tikzpicture}[y=.1cm, x=.1cm,font=\sffamily]
\draw (0,0) -- (-10,-10) node [near end,above=11pt,fill=white] {\tiny $\mathrm{mp}_2$};
\draw (0,0) -- (10,-10) node [near end,above=11pt,fill=white] {\tiny $\mathrm{mp}_1$}; 

\draw[dashed] (10,-10) -- (20,-20);

\draw (20,-20) -- (30,-30) node [near end,above=11pt,fill=white] {\tiny $\mathrm{mp}_1$};

\filldraw[fill=black!40,draw=black!80] (0,0) circle (3pt)    node[anchor=south] {$\mathrm{nf}(s)$};

\filldraw[fill=black!40,draw=black!80] (-10,-10) circle (3pt);    
\filldraw[fill=black!40,draw=black!80] (10,-10) circle (3pt);    
\filldraw[fill=black!40,draw=black!80] (20,-20) circle (3pt);
\filldraw[fill=black!40,draw=black!80] (30,-30) circle (3pt);

\draw (50,0) -- (50,-10) node [midway,right=2pt,fill=white] {\tiny $\mathrm{mp}_1$};
\draw[dashed] (50,-10) -- (50,-20);
\draw (50,-20) -- (50,-30) node [midway,right=2pt,fill=white] {\tiny $\mathrm{mp}_1$};

\filldraw[fill=black!40,draw=black!80] (50,0) circle (3pt)    node[anchor=south] {$\mathrm{nf}(s)$};
\filldraw[fill=black!40,draw=black!80] (50,-10) circle (3pt);    
\filldraw[fill=black!40,draw=black!80] (50,-20) circle (3pt);
\filldraw[fill=black!40,draw=black!80] (50,-30) circle (3pt);
\end{tikzpicture}
\end{center}

\textbf{First main idea: we only can apply $\mathbf{mp}_2$ once, and after applying $\mathbf{mp}_2$ we can not use $\mathbf{mp}_1$ anymore} (corresponds with the figure on the left in the above picture). When we apply $\mathrm{mp}_2$ to one of the elements in $F$, we obtain $(x',y',0)$ with $x'\leq 1$. So we can not apply $\mathrm{mp}_2$ again because the first coordinate is now either $x'=0$ or $x'=1$, which is less than $\frac{3n-5}2$. Then, we can not apply operation $\mathrm{mp}_1$ anymore because $x'\leq 1$ and $z'=0$. So we can apply $\mathrm{mp}_2$ no more than one time and only for the elements $s\in \Ap(S,\F(T))$ that have the first coordinate of $\mathrm{nf}(s)$ either $x=\frac{3n-3}2$ or $x=\frac{3n-5}2$. 

\textbf{Second main idea: after applying $\mathbf{mp}_1$ we cannot use $\mathbf{mp}_2$} (associated to the figure on the right in the above picture). If the first coordinate of $\mathrm{nf}(s)$ is different from $x=\frac{3n-3}2$ or $x=\frac{3n-5}2$, then we can only apply $\mathrm{mp}_1$ to obtain a new factorization. Afterwards we can not apply $\mathrm{mp}_2$ anymore to any new factorization obtained applying operation $\mathrm{mp}_1$. This is because each time $\mathrm{mp}_1$ is applied, the first coordinate of $ \mathrm{nf}(s)$ decreases.

So, in order to count all possible factorizations obtained from $\mathrm{nf}(s)$ with $s\in \Ap(S,\F(T))$, we must enumerate how many times we can apply operation $\mathrm{mp}_1$ first. Then, whenever the first coordinate is $x\geq \frac{3n-5}2$, we can construct an extra factorization with the use of $\mathrm{mp}_2$.

Thus, for every element in $F$, we try to apply operation $\mathrm{mp}_2$ to $\mathrm{nf}(s)$ and then, we try to apply operation $\mathrm{mp}_1$ to $\mathrm{nf}(s)$ as many times as possible.

To justify the five cases in the statement, we show first that when $y<i-2$ there are not $i$ factorizations. So take $(x,y,z)\in F$ with $y<i-2$. Recall that we can apply $\mathrm{mp}_2$ once, at most. When applying operation $\mathrm{mp}_1$, we subtract one on the second coordinate. So in this settings, we can apply at most $i-3$ times $\mathrm{mp}_1$ and, eventually, one more time $\mathrm{mp}_2$. This process adjoins up to $i-3+1=i-2$ new factorizations, that added to the original $(x,y,z)$, cannot give $i$ factorizations.

Now we consider $s\in M_i$ with $\mathrm{nf}(s)=(x,y,z)$.
\begin{itemize}
\item If the second coordinate is $y=i-2$, as we have seen above, we can obtain $i-2$ new factorizations with $\mathrm{mp}_1$. So to get exactly $i$ factorizations, we need an extra one by applying $\mathrm{mp}_2$. But this is possible if and only if $z=1$ and $x\in\left\{\frac{3n-3}2, \frac{3n-5}2\right\}$. In these cases, as $\frac{n-3}{2}\geq i$, we get $x\geq\frac{3n-5}{2}>\frac{3n-9}{2}\geq 3i$, and the first coordinate is large enough to subtract $(i-2)$ times 3. This corresponds with \textit{1)}.

\item Now we assume $y=i-1$. We can construct $i-1$ new factorizations by applying the $\mathrm{mp}_1$ operation and, adding the original one, we obtain $i$ factorizations. So, it is necessary that operation $\mathrm{mp}_2$ can not be applied. Hence, we have $z=0$ and $x\leq \frac{3n-3}2$ (when $(x,y,z)\in F_1$), or  $z=1$ and $x\leq \frac{3n-7}2$ (when $(x,y,z)\in F_3$) as in this case, if $x\geq \frac{3n-5}2$ we can apply $\mathrm{mp}_2$. In both cases, we also need that $x\geq 3(i-1)$. This yields \textit{2)} and \textit{3)}.


\item If we take $y\geq i$, we need to have $x\leq 3i-1$ to ensure that we only can apply $\mathrm{mp}_1$  $i-1$ times. As $x\leq 3i-1$ and $i\leq \frac{n-3}2$, we have $x\leq 3i-1\leq \frac{3n-11}{2}$, so operation $\mathrm{mp}_2$ can not be applied. In these cases, we will need again the extra condition $x\geq 3(i-1)$. Finally, recall that the case $y=\frac{n-3}{2}$ and $z=1$ is not in the Ap\'ery set $\Ap(S,\F(T))$. For this reason, in \textit{5)}, there is one element less. This yields \textit{4)} and \textit{5)}.\qedhere

\end{itemize}
\end{proof}
\begin{remark}
We left the proofs of the following curious facts to the reader.
\begin{enumerate}[1)]\label{curiosities}
\item $M_i=\emptyset$ for $i\ge \frac{n+1}2$. 
\item $M_i$ has $6n-12i-1$ elements for $i\le \frac{n-1}2$.
\item $\sum_{i\in \mathbb N} \# M_i=\F(T)$.
\end{enumerate}
\end{remark}

\subsection{Restrictions on the construction of \textsf L-shapes}
Let $L$ be an \textsf L-shape associated with $S$. Conditions (C1) and (C2) imply that, if $f$ is a factorization of $s\in S$ appearing in $L$, then any $f'\le f$  corresponds to a factorization of an element $s'\in S$ (actually in the Ap\'ery set of $\F(T)$). Moreover, $f'$ is the only factorization of $s'$ occurring in $L$.

With this idea in mind, we start showing that the minimal elements in $\mathsf Z(M_i)$ are enough to ``control" all the elements in $\mathsf Z(M_i)$.
  
\begin{lemma}\label{caract_nf}
Let $s\in \Ap(S,\F(T))$ and let $(\alpha,\beta,\gamma)\in \mathsf Z(s)$. Then $\mathrm{nf}(s)=(\alpha,\beta,\gamma)$ if and only if $\beta\le\frac{n-1}{2}$ and $\gamma\in\{0,1\}$.
\end{lemma}
\begin{proof}
This statement follows from lemmas \ref{fact-en-ap} and \ref{wholeApery}.
\end{proof}


\begin{lemma}\label{nfaresufficient}
Let $i$ be a positive integer. If we take $s,s'\in M_i$, then the following facts are equivalent:
\begin{enumerate}[1)]
\item $\mathrm{nf}(s')< \mathrm{nf}(s)$,
\item for any factorization $\zeta=(\alpha,\beta,\gamma)\in \mathsf Z(s)$ there exists a unique $\zeta'=(\alpha',\beta',\gamma')\in \mathsf Z(s')$ such that $\zeta'< \zeta$.
\end{enumerate}
\end{lemma}
\begin{proof}
\emph{2)} implies \emph{1)} follows easily from the characterization of normal form given in Lemma \ref{caract_nf}.

Let us see now that statement \emph{1)} implies \emph{2)}. Write $\mathrm{nf}(s)=\mathrm{nf}(s')+(a,b,c)$, with $(a,b,c)\in \mathbb N^3\setminus\{(0,0,0)\}$. It follows that $(a,b,c)+\mathsf Z(s')\subseteq \mathsf Z(s)$. As both sets have the same cardinality, $i$, we obtain an equality. Assertion \emph{2)} now follows easily.
\end{proof}

Let $\mathcal M_i= \{ s\in M_i\mid \mathrm{nf}(s) \hbox{ is minimal in } \mathrm{nf}(M_i)\}$. From Lemma~\ref{nfaresufficient} and (C2), in order to construct an \textsf L-shape for every possible $i$, we only have to choose a factorization for each $s\in \mathcal M_i$.


The following result gives these minimal elements.

\begin{lemma}
$\mathrm{Minimals}_\le(\mathrm{nf}(M_i))=\left\{\left(\frac{3n-5}2,i-2,1\right),(3(i-1),i-1,0)\right\}$.
\end{lemma}
\begin{proof}
From Lemma \ref{casesgeneral}, 
\begin{multline*}
\mathrm{Minimals}_\le(\mathrm{nf}(M_i))=\mathrm{Minimals}_\le\Bigg\{\Bigg(\frac{3n-5}2,i-2,1\Bigg), (3(i-1),i-1,0),\\ 
 (3(i-1),i-1,1), (3(i-1),i,0),(3(i-1),i,1)\Bigg\},
\end{multline*} and the statement follows.
\end{proof}


\begin{corollary}
Each $\mathcal M_i$ has two elements.
\[
\mathcal M_i=\left\{\mathbf s_i=\frac{3n-5}2n+(i-2)(3n-2)+3n-1,\ \mathbf s'_i=3(i-1)n+(i-1)(3n-2)\right\}.
\]
\end{corollary}

Next lemma gives an important reduction for the construction of \textsf L-shapes. By using this result, it is only necessary to choose a factorization involved in $\mathrm{mp}_1$, that is, $(3,1,0)$ or $(0,0,2)$, to control all $\mathbf s'_i\in\mathcal M_i$, $i\in \{3,\ldots ,\frac{n-1}2\}$.
\begin{lemma}\label{restriccion-primas}
Let $L$ be an \textsf L-shape associated with $S=\langle n,3n-2,3n-1,F(T)\rangle$, and let $\mathbf{s}=3n+(3n-2)=2(3n-1)$ and $i\in\{3,\ldots, \frac{n-1}2\}$.
\begin{enumerate}[1)]
\item $\mathsf Z(\mathbf s)=\{(3,1,0),(0,0,2)\}$.
\item If $(3,1,0)\in L$, then $L\cap\mathsf Z(\mathbf s_i')=\{(3(i-1),i-1,0)\}$.
\item If $(0,0,2)\in L$, then $L\cap\mathsf Z(\mathbf s_i')=\{(0,0,2(i-1))\}$.
\end{enumerate}
In particular, $L\cap\mathsf Z(\mathbf s_i')\subset \{(3(i-1),i-1,0),(0,0,2(i-1))\}$.
\end{lemma}
\begin{proof}
The first assertion is already known, and it follows from \eqref{minimalpresentation} and the proof of \cite[Theorem 10.25]{ns}.

The factorizations of the element $\mathbf{s'_i}$ are $(3x,x,2(i-1-x))$ for $x\in\{0,\ldots ,i-1\}$. Take $x\in\{1,\ldots,i-1\}$. If $(3x,x,2(i-1-x))\in L$, then $(3,1,0)\le (3x,x,2(i-1-x))$ and $(0,0,2)\le (3x,x,2(i-1-x))$. Hence Condition (C2) asserts that 
$(3,1,0),(0,0,2)\in L$, and thus $\#\mathsf Z(\mathbf{s})\cap L=2$, contradicting (C1). This means that none of these factorizations can be in $L$. Assertions \emph{2)} and \emph{3)} now follow easily by taking $x=1$ and $x=i-1$.
\end{proof}

Notice that $\mathbf s=\mathbf s_2'$. 
\begin{lemma}\label{factorizations}
Let $L$ be an \textsf L-shape associated with $S=\langle n,3n-2,3n-1,F(T)\rangle$, and let $i\in \{3,\ldots,\frac{n-1}2\}$. Then $L\cap \mathsf Z(\mathbf s_i)\subset \left\{ \left(\frac{3n-5}2,i-2,1\right),\left(0,\frac{n-3}2+i,0\right), \left(\frac{3n-5}2-3(i-2),0,2(i-2)+1\right)\right\}$. 
\end{lemma}
\begin{proof}
The other factorizations of $\mathbf s_i$ are $\left(\frac{3n-5}2-3x,i-2-x,2x+1\right)$ with $x\in \{1,\ldots, i-3\}$. Arguing as in Lemma~\ref{restriccion-primas}, we deduce that they cannot be in $L$. 
\end{proof}

\begin{remark}
When $i=2,$ we have only two possible choices: $\left(0,\frac{n+1}2,0\right)$ and  $\left(\frac{3n-5}2,0,1\right).$ The first and third factorizations given by Lemma~\ref{factorizations} are the same.
\end{remark}

\subsection{A family of \textsf L-shapes} 
Now we will try to put, in an ordered way, the different possible factorizations for each element of $\{\mathbf{s}\}\cup\{\mathbf s_i\mid i\in\{2,\ldots ,\frac{n-1}2\}\}$. Such possible factorizations are given by Lemma \ref{factorizations}.

For $i=\frac{n-1}2$, the element $\mathbf s_{\frac{n-1}2}= (n-2)(3n-2)$ has the following three factorizations to choose:
\[
\left(0,n-2,0 \right), \left(\frac{3n-5}2,\frac{n-5}2,1 \right), \left(5,0,n-4 \right).
\]

\begin{enumerate}[(a)]
\item The factorization $(0,n-2,0)$ is over  $\left(0,\frac{n-3}2+i,0 \right)$ for all possible $i$. So, when choosing $(0,n-2,0)$, the $\mathsf{L}$-shape $L$ must contain the factorization $(0,\frac{n-3}2+i,0)$ of the other possible $\mathbf s_i$'s. We can choose for $\mathbf{s}$ both of its factorizations, that is, either $(0,0,2)$ or $(3,1,0)$. Therefore, we can construct two different $\mathsf L$-shapes from $(0,n-2,0)$. The choices for $\mathbf s_i$ and $\mathbf s_i'$ are respectively (recall that the rest of factorizations of elements in $\Ap(S,\F(T))$ are forced by them):
\begin{multline*}
\quad \left(0,\frac{n+1}2,0\right),\ldots ,\left(0,\frac{n-3}2+i,0\right), \ldots, (0,n-2,0),\\
(0,0,2), \ldots,(0,0,2(i-1)),\ldots, (0,0,n-3),
\end{multline*}
or
\begin{multline*}
\quad \left(0,\frac{n+1}2,0\right),\ldots ,\left(0,\frac{n-3}2+i,0\right), \ldots, (0,n-2,0),\\
(3,1,0), \ldots,(3(i-1),i-1,0),\ldots, \left(3\frac{n-3}2,\frac{n-3}2,0\right).
\end{multline*}

\item If we choose $\left(\frac{3n-5}2,\frac{n-5}2,1 \right)$ as the factorization of $\mathbf s_{\frac{n-2}2}$ in $L$,  then all the other elements must be put in $\left(\frac{3n-5}2,i-2,1 \right)$. Even $\mathbf{s}$ must be put in $(3,1,0)$ because it is under $\left(\frac{3n-5}2,\frac{n-5}2,1 \right)$. So we only can construct one $\mathsf L$-shape from $\left(\frac{3n-5}2,\frac{n-5}2,1 \right)$ (this construction corresponds to elements in $F$). The choices for $\mathbf s_i$ and $\mathbf s_i'$ are respectively
\begin{multline*}
\quad \left(\frac{3n-5}2,0,1\right),\ldots , \left(\frac{3n-5}2,i-2,1 \right), \ldots , \left(\frac{3n-5}2,\frac{n-5}2,1 \right),\\
(3,1,0), \ldots,(3(i-1),i-1,0),\ldots, \left(3\frac{n-3}2,\frac{n-3}2,0\right).
\end{multline*}
\item 
Finally, if we choose $(5,0,n-4)$ for the last element $\mathbf s_{\frac{n-1}2}$, we can not take $(3,1,0)$ for $\mathbf{s}$, because $(0,0,2)$ is below $(5,0,n-4)$. So, as we cannot choose $(3,1,0)$ for $\mathbf{s}$, we also cannot select $\left(\frac{3n-5}2,i-2,1 \right)$ for $\mathbf{s_i}$, $i\in\{2,\ldots,\frac{n-1}2\}$, since $(3,1,0)$ is below these elements. Hence we can only choose for $\mathbf s_\frac{n-3}2=(n-3)(3n-2)$ two factorizations: $(0,n-3,0)$ and $(8,0,n-6)$. Both of them are feasible. Now, if we choose $(0,n-3,0)$, the remaining elements are determined by this selection, as in (a). However, if we choose $(8,0,n-6)$, we have two new possibilities. Therefore, we obtain a new \textsf{L}-shape in each step. There are $\frac{n-5}2$ steps, so we have $\frac{n-3}2$ new \textsf L-shapes.  The choices for $\mathbf s_i$ and $\mathbf s_i'$ are respectively:
\[
\begin{array}{ll}
j=0, & \left(0,\frac{n+1}2,0\right),\ldots, \left(0,\frac{n-3}2+i,0\right),\ldots,(0,n-4,0), (0,n-3,0),(5,0,n-4),\\
& \hfill (0,0,2), \ldots,(0,0,2(i-1)),\ldots, (0,0,n-3);\\
j=1, &\left(0,\frac{n+1}2,0\right),\ldots, \left(0,\frac{n-3}2+i,0\right),\ldots, (0,n-4,0),(8,0,n-6),(5,0,n-4),\\
& \hfill (0,0,2), \ldots,(0,0,2(i-1)),\ldots, (0,0,n-3);\\
\vdots & \vdots \\[3mm]
j=\frac{n-5}2, & \left(\frac{3n-5}2,0,1\right),\ldots ,\left(\frac{3n-5}2-3(i-2),0,2(i-2)+1\right),\ldots ,(8,0,n-6),(5,0,n-4),\\
& \hfill (0,0,2), \ldots,(0,0,2(i-1)),\ldots, (0,0,n-3).
\end{array}
\]
\end{enumerate}

Summarizing, we can obtain two \textsf L-shapes from $(0,n-2,0)$, one more from $\left(\frac{3n-5}2,\frac{n-5}2,1 \right)$, and $\frac{n-3}2$ new \textsf L-shapes from $(5,0,n-4)$. So, we can construct $2+1+\frac{n-3}2=\frac{n+3}{2}$ different \textsf L-shapes for $\langle n,3n-2,3n-1,\mathrm F(T)\rangle$. 

\begin{theorem}\label{teo:final}
The number of different \textsf L-shapes for an embedding dimension four numerical semigroup is not upper bounded.
\end{theorem}

\begin{proof}
Consider $T_m=\langle m,3m-2,3m-1\rangle$ and $S_m=\langle m,3m-2,3m-1,\mathrm{F}(T_m)\rangle$ for odd $m\geq5$. For any integer $n$ greater than four, we can take $T_{2n-3}=\langle 2n-3,3(2n-3)-2,3(2n-3)-1\rangle$ and $S_{2n-3}=\langle 2n-3,3(2n-3)-2,3(2n-3)-1,\mathrm F(T_{2n-3})\rangle$. There exists, at least, $\frac{(2n-3)+3}2=n$ different \textsf L-shapes associated to $S_{2n-3}$. 
\end{proof}
Figure \ref{ejemplo-17} shows the \textsf L-shapes obtained for $m=17$ and $n=10$.

\begin{figure}
\includegraphics{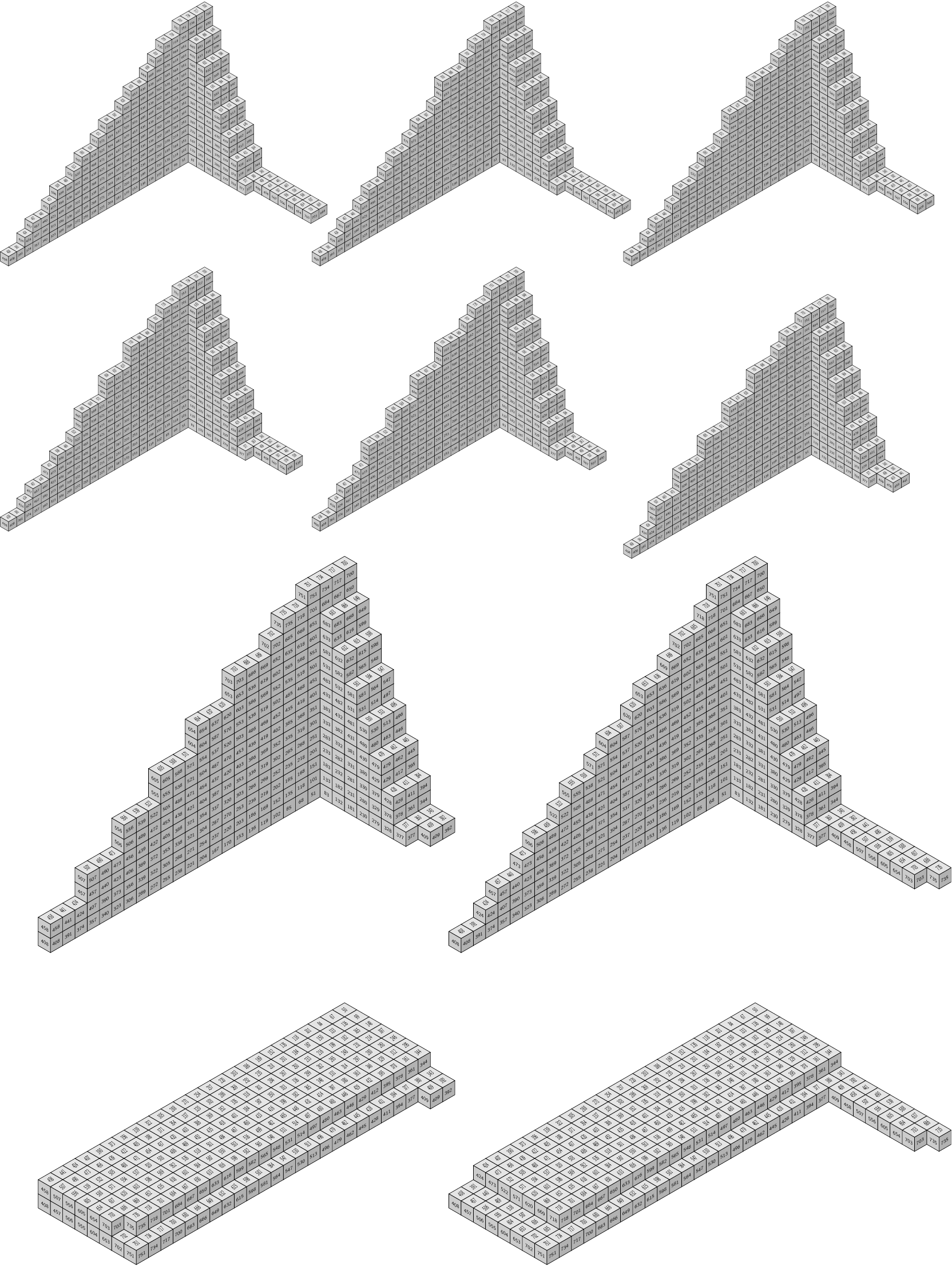}
\caption{The set of \texttt L-shapes for $S=\langle 17, 49, 50, 376 \rangle$}
\label{ejemplo-17}
\end{figure}	
\begin{remark}
From Lemma \ref{casesgeneral}, we can obtain the maximal elements in each $M_i$. From these elements we can obtain the pseudo-Frobenius numbers of $S$ because pseudo-Frobenius numbers are the maximal elements in the set $\mathrm{Ap}(S_{2n-3},\mathrm{F}(T_{2n-3}))-\mathrm{F}(T_{2n-3})=\{w-\mathrm{F}(T_{2n-3})\mid~w\in\mathrm{Ap}(S_{2n-3},\mathrm{F}(T_{2n-3}))\}$, with respect to $\le_S$. One can easily deduce that 
\[
\mathrm{PF}(S)=\left\{\frac{3n^2-9n+4}2, \frac{3n^2-7n+2}2, \frac{3n^2-13n+8}2\right\}.
\]
\end{remark}

\begin{remark}
Theorem~\ref{teo:final} has his counterpart for weighted Cayley digraphs of degree three. More precisely, the digraph $C(\mathrm{F}(T_{2n-3});2n-3,6n-9,6n-10;2n-3,6n-9,6n-10)$, for $n\geq4$,  have related $n$ minimum distance diagrams.
\end{remark}

\section{Aknowledgements}
Authors thank the support of MTM2011-28800-C02-01, 2009SGR1387, MTM2010-15595, FQM-343, FQM-5849, MTM2010-15595, FQM-343 and FEDER funds. Part of this work was done during a visit of the second author to the Universidad de Almer\'{\i}a supported by the `plan propio' of this university.

\end{document}